\documentclass[12pt]{article}
\usepackage{geometry}                
\usepackage{graphicx,color}
\usepackage{amssymb,amsmath,amsthm,mathrsfs}
\usepackage[all,cmtip]{xy}
\usepackage{epstopdf, comment}

\DeclareMathOperator{\diam}{diam }

\DeclareMathOperator{\chem}{chem }

\DeclareMathOperator{\Mod}{Mod }

\DeclareMathOperator{\dist}{dist }

\DeclareMathOperator{\Jac}{Jac}

\DeclareMathOperator{\carr}{carr}
\DeclareMathOperator{\cont}{c}
\DeclareMathOperator{\disc}{d}

\numberwithin{equation}{section}

\newtheorem{theorem}{Theorem}[section]

\newtheorem{lemma}[theorem]{Lemma}
\newtheorem{corollary}[theorem]{Corollary}

\theoremstyle{remark}
\newtheorem*{remark}{Remark}
\newtheorem*{modifications}{Modifications}

\theoremstyle{definition}

\begin{document}

\title{Homogenization of random quasiconformal mappings and random Delauney triangulations}
\author{Oleg Ivrii and Vladimir Markovi\'c}
\date{May 20, 2019}
\maketitle

\abstract{
In this paper, we solve two problems dealing with the homogenization of random media. We show that a random quasiconformal mapping is close to an affine mapping, while a circle packing of a random Delauney triangulation is close to a conformal map, confirming a conjecture of K.~Stephenson. We also show that on a Riemann surface equipped with a conformal metric, a random Delauney triangulation is  close to being  circle packed.
}

\section{Introduction}

\subsection{Random quasiconformal mappings}

Our model of a random quasiconformal mapping depends on a probability measure $\lambda$ on the unit disk $\mathbb{D} = \{z \in \mathbb{C}: |z| < 1\}$.
 For each cell in a square grid in the complex plane, randomly assign a complex number  in the unit disk according to the measure $\lambda$. The collection of these numbers defines a Beltrami coefficient on $\mathbb{C}$ which is constant on the cells of the grid.
 We write $ w^{\mu}$ for the unique homeomorphism of the plane which solves the Beltrami equation
 $$
 \overline{\partial} w(z) = \mu(z) \partial w(z)
 $$
and fixes $0, 1, \infty$. If $\lambda$ is supported on a compact subset of the unit disk, the existence of $ w^{\mu}$ is guaranteed by the measurable Riemann mapping theorem. In general,  $ w^{\mu}(z)$ exists and is unique by virtue of  $K_\mu = \frac{1 + |\mu|}{1-|\mu|}$ being locally bounded, although  it may not be surjective. We refer to $w^{\mu}$ as a random quasiconformal mapping, even though it may not be a genuine quasiconformal mapping.
Our first main theorem says that if the mesh size $\delta > 0$ is small, then with high probability, $ w^{\mu}$ is close to an affine transformation $A_{\lambda} = w^{\mu_\lambda}$ determined by the measure $\lambda$.
(Since an affine map has constant dilatation, $\mu_\lambda$ is a constant with absolute value less than 1.)

\begin{theorem}
\label{main-thm}
Suppose $\lambda$ is a probability measure on the unit disk. For each cell  of a square grid in the plane of mesh size $\delta$, randomly assign a complex number in the unit disk according to the measure $\lambda$.
There exists an affine transformation $A_\lambda$ such that for any compact set $K \subset \mathbb{C}$ and $\varepsilon > 0$, when $\delta \le \delta_0(K, \varepsilon)$ is sufficiently small,
$\|  w^{\mu} -  A_{\lambda}\|_{C(K)} < \varepsilon$ holds with probability
at least $1-\varepsilon$.
\end{theorem}

We do not know how to explicitly determine the affine map $A_\lambda$ for general $\lambda$. However, if $\lambda$ respects the $90^\circ$ symmetry of the model, then $A_{\lambda}$ is the identity mapping:

\begin{corollary}
If  $d\lambda(z) = d\lambda(-z)$, then $A_{\lambda}(z) = z$ is the identity map.
\end{corollary}

\begin{proof}
Observe that $ w^{\mu}(iz)/ w^{\mu}(i)$ is the normalized quasiconformal mapping with dilatation $- \mu(iz)$. Since the square grid is invariant under multiplication by $i$, the random quasiconformal maps
 $ w^{\mu}(iz)/ w^{\mu}(i)$ and $ w^\mu(z)$ are equally likely. Therefore, the affine transformation $A_{\lambda}(z)$ satisfies the relation $A_\lambda(iz)/A_\lambda(i) = A_\lambda(z)$ which forces $A_\lambda(z) = z$.
\end{proof}

In an unpublished manuscript  \cite{ARST}, K.~Astala, S.~Rohde, E.~Saksman and T.~Tao gave a different proof of Theorem \ref{main-thm} for random quasiconformal mappings with uniformly bounded distortion. Their method is based on the  homogenization of iterated singular integrals.
By contrast, our proof is  more elementary: it is based on the geometric definition of quasiconformal mappings.

 We first show that with high probability,
a random quasiconformal map is roughly quasiconformal, that is, it stretches the moduli of all rectangles whose sides have length $\ge \varepsilon$ by a bounded amount. This uses a simple lemma about percolation on the square grid that we have learned
from a paper of P.~Mathieu \cite{mathieu} on random walk in random environments.

We then show that there exists an extremal direction such that with positive probability, the
 random quasiconformal map stretches moduli of squares in that direction by approximately the maximal amount. We  promote {\em positive probability} to {\em high probability} by subdividing a square that is stretched by approximately the maximal amount into a large number of small squares. On one hand, the moduli of the images of the small squares are independent random variables since they are disjoint, while one the other hand, the extremality of the big square forces all small squares to be extremal.

The above argument shows that there is a sequence of good scales $\delta_k \to 0$, such that with high probability, the random quasiconformal mapping $w^\mu$ constructed using the square grid of mesh size $\delta_k$ is close to an affine mapping  on any fixed compact subset of the plane. To show that all sufficiently small scales are good, we use the following principle: if an orientation-preserving homeomorphism is  conformal off a random set of small measure, then it is close to a conformal map.

\subsection{Random Delauney triangulations}

A {\em circle packing} $\mathcal P = \{C_i\}$ is a collection of circles in the plane with disjoint interiors. The {\em tangency pattern} of $\mathcal P$  is an embedded graph in the plane whose vertices are centers of circles in $\mathcal P$ and edges are line segments which connect centers of tangent circles.
The Koebe-Andreev-Thurston Circle Packing Theorem \cite{koebe, thurston} says that  any  finite triangulation $\mathcal T$ of a topological disk admits a {\em maximal circle packing} $\mathcal P = \bigcup C_i \subset \mathbb{D}$ whose boundary circles are horocycles. Furthermore, this maximal packing is unique up to M\"obius transformations.

For a discrete set of points $V$ in the plane, the {\em Voronoi tessellation} is a decomposition of the complex plane $\mathbb{C} = \bigcup F_x$ where
$F_x$ consists of all points $z \in \mathbb{C}$ for which $\min_{y \in V} |y - z| = |x-z|$. Each $F_x$ is a polygon, although it could be unbounded.
If the points in $V$ are in general position, that is, if no three points lie on a line and no four points lie on a circle, then one can define the {\em Delauney triangulation} as the dual graph to the Voronoi tessellation. Namely, its vertex set is $V$, and there is an edge between $x$ and $y$ if the intersection $F_x \cap F_y$ is a line
segment. Since the union of all triangles in a Delauney triangulation is the convex hull of $V$, it is a  topological disk.

Let $\Omega$ be a simply-connected domain whose boundary is a $C^1$ curve. Randomly choose $N \ge 1$ points in $\Omega$ with respect to Lebesgue measure and take the union of the Delauney triangles contained in $\Omega$. Based on numerical experiments, Kenneth Stephenson suggested that when $N \ge 1$ is large, then  with high probability, the maximal circle packing of a random Delauney triangulation approximates a conformal map $\varphi: \Omega \to  \mathbb{D}$. In this paper, we prove this conjecture.
To be precise, fix two points $z_1, z_2 \in \Omega$. For each $i=1,2$, let $v_i \in \mathcal{T}$ be the closest point to $z_i$ (in case of a tie, choose $v_i$ arbitrarily).
Let $\mathcal P$ be the maximal circle packing of $\mathcal T$ normalized so that the $C_{v_1}$ is centered at the origin while the center of $C_{v_2}$ lies on $(0,1)$.
The {\em circle packing map} $\varphi_{\mathcal P}: \carr \mathcal T \to \carr \mathcal P$ is the piecewise linear map that takes points of $\mathcal T$ to centers of circles
and is linear on triangles. Here, the {\em carrier} of a triangulation is simply the union of all the triangles in the triangulation.

\begin{theorem}
\label{main-thm4}
Let $\Omega$ be a bounded simply-connected domain in the plane with $C^1$ boundary and $\varphi: \Omega \to \mathbb{D}$ be the conformal map with
 $\varphi(z_1) = 0$ and $0 < \varphi(z_2) < 1$. Consider the random Delauney triangulation with $N$ points.
For any compact set $K \subset \Omega$ and $\varepsilon > 0$, when $N \ge N_0(K, \varepsilon)$ is sufficiently large,  $\|\varphi_{\mathcal P} - \varphi\|_{C(K)} < \varepsilon$ holds with probability
at least $1-\varepsilon$.
\end{theorem}

 The above theorem remains true if the random set of points is generated using a  Poisson point process of high intensity.

 The proof of Theorem \ref{main-thm4} is similar to that of Theorem \ref{main-thm} in that most of the work goes into showing that with high probability, the circle packing map $\varphi_{\mathcal P}$ is roughly quasiconformal.
 One first shows that with high probability,
the discrete modulus of every rectangle $R \subset \Omega$ whose sides have length $\ge \varepsilon$ is bounded from above and below. The discrete modulus is simple to analyze since it only depends on the combinatorics of the Delauney triangulation in $R$. In general, the discrete and continuous moduli of $\varphi_{\mathcal P}(R)$ are unrelated, however, if the triangulation in question has bounded valence, the two notions of modulus agree up to a multiplicative constant. While a random Delauney triangulation may have vertices of arbitrarily large valence, they are quite rare and can be ``avoided'' using a percolation argument.

At this point, we are presented with a second difficulty.
In the setting of random quasiconformal mappings, the modulus of $w^{\mu}(R)$ only depends on the Beltrami coefficient on $R$, however, in a circle packing, the modulus of $\varphi_{\mathcal P}(R)$ also depends on the behaviour of the triangulation outside of $R$. However, if all circles in the packing have small radii, then by a fundamental result of He and Schramm \cite{he-schramm}, $\Mod \varphi_{\mathcal P}(R)$ is  determined by the combinatorics of the triangulation in $R$ up to small error which tends to 0 as the radii of the circles in the packing shrink.

\begin{remark}
Let $\Sigma$ be a Riemann surface of genus $g \ge 2$. Consider the random Delauney triangulation on $\Sigma$ with respect to the hyperbolic metric. According to \cite[Proposition 9.1]{stephenson-book}, the maximal circle packing will live on a Riemann surface $\Sigma_{\mathcal P}$ homeomorphic to $\Sigma$, however, the complex structure may be different. Using the methods of this paper, one can show that when the number of points $N$ is very large, then with high probability, the Riemann surface
$\Sigma_{\mathcal P}$ is close to $\Sigma$ in the Teichm\"uller space $\mathcal T_g$ of Riemann surfaces of genus $g$, and furthermore, the mapping $\varphi_{\mathcal P}$ is uniformly close to the Teichm\"uller map from $\Sigma \to \Sigma_{\mathcal P}$.
\end{remark}

\subsection{Random walk in random environments}
\label{sec:rwre}

For comparison, we mention some results about random walk in random environments.
Let $\lambda$ be a probability measure on $(0, \infty)$. For each edge in the square grid $\mathbb{Z}^2$, randomly choose its conductance according to $\lambda$. Let $S_n$ be the random walk in $\mathbb{Z}^2$ which starts at the origin, and at each step, the walker moves from a vertex $x$ to an adjacent vertex $y \sim x$ with probability $$\frac{c(x,y)}{\sum_{z \sim x} c(x,z)}.$$  In 2004, Sidoravicius and Sznitman \cite{sidoravicius-sznitman} showed that if the conductances are uniformly bounded away from zero and infinity, then $S_n/\sqrt{n}$ converges to Brownian motion, as in the unweighted case.
Several years later, the case of {\em arbitrary} non-zero conductances was resolved independently by P.~Mathieu \cite{mathieu} and M.~Biskup and T. Prescott \cite{BP}. For a survey on the random conductance model, see \cite{biskup-survey}.

The model of random quasiconformal maps can be interpreted as a continuous analogue of simple random walk in random media where one simulates Brownian motion in a random environment: in each cell of the square grid, Brownian motion is to be stretched in some direction depending on the dilatation. Essentially, this process simulates the image of Brownian motion under the quasiconformal map. This has been studied by Osada \cite{osada} under the name of {\em homogenization of diffusing processes}\/, although he only discussed the case of bounded distortion.

\section{Moduli of curve families}

We will work with two notions of moduli of curves. In the continuous setting, a {\em metric} $\rho(z)$ is a non-negative measurable function defined on a domain $\Omega \subset \mathbb{C}$.
 The  area of $\rho$ is defined as
$$
A(\rho) = \int_{\Omega} \rho^2 |dz|^2.
$$
One can use $\rho$ to measure lengths of rectifiable curves:
$$
\ell_{\rho}(\gamma) = \int_\gamma \rho |dz|.
$$
A metric is said to be {\em admissible} for a family of rectifiable curves $\Gamma$ (contained in $\Omega$) if the $\rho$-length of every curve $\gamma \in \Gamma$ is at least 1. The {\em modulus} of the curve family $\Gamma$ is defined as
$$
\Mod \Gamma := \inf_\rho A(\rho),
$$
where the infimum is taken over all admissible metrics $\rho$.
By a  {\em conformal rectangle} $R$, we mean a Jordan domain in the plane with four marked boundary points. In this paper, all conformal rectangles will be marked, i.e.~equipped with a  distinguished pair of opposite sides.
 Let $\Gamma_{\leftrightarrow}$ be the family of curves connecting the distinguished pair of opposite sides of $R$ and $\Gamma_{\updownarrow}$ denote the conjugate family which connects the other pair of opposite sides. It is well known that
\begin{equation}
\label{eq:product-of-moduli}
\Mod \Gamma_{\leftrightarrow} \cdot \Mod \Gamma_{\updownarrow} = 1.
\end{equation}

Given a (geometric) rectangle $R$, we denote the length of its marked sides by $\ell_1(R)$ and the length of the unmarked sides by $\ell_2(R)$. Then, $\Mod R = \Mod \Gamma_{\leftrightarrow} = \ell_1(R)/\ell_2(R)$. We denote the side length of a square $S$ by $\ell(S)$. All squares have modulus 1.

For two compact sets $E, F \subset \mathbb{C}$, the {\em Hausdorff distance} $d(E,F)$ is defined as the minimal number $t \ge 0$ such that any point of $E$ is within $t$ of some point of $F$ and vice versa. To define the Hausdorff distance between two conformal rectangles, one also needs to make sure that the marked sides line up.

It is easy to see that modulus of a conformal rectangle varies continuously in the Hausdorff topology. The following lemma says that the modulus of the image of a conformal rectangle under a quasiconformal map  does not change much under small perturbations:
\begin{lemma}
\label{continuity-of-moduli}
Suppose $f: \mathbb{C} \to \mathbb{C}$ is a $K$-quasiconformal mapping and $S, S'$ are two squares in the plane. For any $\varepsilon > 0$, there exists a $\delta = \delta(\varepsilon, K) > 0$, so that if the relative Hausdorff distance $d(S, S')/\ell(S) < \delta$, then $|\Mod f(S') - \Mod f(S)| < \varepsilon$.
\end{lemma}

In the discrete setting, a {\em metric} $\rho(v)$ is a function on the vertices of a planar graph $G$. The area of $\rho$ is defined as
$$
A(\rho) = \sum_{v \in G} \rho(v)^2.
$$
A path $\gamma = \langle x_0, x_1, x_2, \dots, x_n \rangle$ is a collection of vertices such that $x_i \sim x_{i+1}$ are connected by an edge. One can use $\rho$ to measure lengths of paths:
$$
\ell_\rho(\gamma) = \sum_{v \in \gamma} \rho(v).
$$
The notions of admissibility and discrete modulus are defined as in the continuous case. By a {\em combinatorial rectangle} $R$, we mean a topological rectangle enclosed by a finite collection of edges of $G$,
 and four vertices of $G$ have been marked on $\partial R$.

\section{Roughly quasiconformal maps}
\label{sec:roughly-qc}

According to the geometric definition of quasiconformality, an orientation-preserving  homeomorphism $f: U \to \mathbb{C}$  is quasiconformal if it distorts moduli of rectangles in $U$ by a bounded amount. It is well known that one can test quasiconformality by looking at round annuli of modulus 2 or at rectangles of modulus 10. For convenience of the reader, we recall the proofs.

\begin{lemma}
\label{testing-annuli}
Suppose an orientation-preserving homeomorphism $f: U \to \mathbb{C}$ distorts moduli of all annuli $A = A(z, r, 2r) \subset U$ by a bounded amount:
\begin{equation}
\label{eq:annuli-qc}
(1/K) \cdot \Mod A \, \le \,  \Mod f(A) \, \le \, K \cdot \Mod A.
\end{equation}
Then, $f$ is $L(K)$ quasiconformal.
\end{lemma}

\begin{proof}
We will show that $f$ satisfies the following quasisymmetry condition: there exists a constant $C(K) > 0$ such that
\begin{equation}
\label{eq:quasisymmetry}
\frac{\sup_{|z - x| = 2r} |f(z) - f(x)|}{\inf_{|z - x| = 2r} |f(z) - f(x)|} \le C, \qquad \text{whenever }B(x, 4r) \subset U.
\end{equation}
Once we show (\ref{eq:quasisymmetry}), the lemma follows from the equivalence of quasiconformality and quasisymmetry, e.g.~see \cite[Theorerm 3.4.1]{AIM}.

Suppose $|y - x| = r$ and $|z - x| = 2r$. Consider the annulus $A = A(x, r, 2r)$. Since $f(A)$ separates $f(x), f(y)$ from $f(z)$ and $\Mod f(A)$ is bounded from below, we have
\begin{align}
\label{eq:quasisymmetry3}
 |f(z) - f(y)| \ge c \, |f(y) - f(x)|, \\
|f(z) - f(x)| \ge c \, |f(y) - f(x)|,
\end{align}
for some constant $c > 0$ which depends only on $K$. It follows that
\begin{equation}
\label{eq:quasisymmetry1}
\inf_{|z - x| = 2r} |f(z) - f(x)| \ge c \, \sup_{|y - x| = r} |f(y) - f(x)|.
\end{equation}
For the reverse inequality, note that if $y$ is the midpoint of $x$ and $z$ then we have the third inequality
$$
|f(x) - f(y)| \ge c \, |f(y) - f(z)|,
$$
in which case,
\begin{align}
\sup_{|z - x| = 2r} |f(z) - f(x)| & \le  \sup_{|z - x| = 2r} \bigl \{ |f(z) - f(y)| + |f(y) - f(x)| \bigr \} \\
\label{eq:quasisymmetry2}
&  \le (1+c^{-1}) \, \sup_{|y - x| = r} |f(y) - f(x)|.
\end{align}
Putting (\ref{eq:quasisymmetry1}) and (\ref{eq:quasisymmetry2}) together completes the proof.
\end{proof}

\begin{lemma}
\label{testing-rectangles}
Suppose an orientation-preserving homeomorphism $f: U \to \mathbb{C}$ distorts moduli of all rectangles with aspect ratio $10$ contained in $U$ by a bounded amount:
\begin{equation}
\label{eq:rectangle-qc}
(1/K) \cdot \Mod R \, \le \,  \Mod f(R) \, \le \, K \cdot \Mod R.
\end{equation}
Then, $f$ is $L(K)$ quasiconformal.
\end{lemma}

\begin{proof}
We follow the argument from Hinkkanen's paper \cite{hinkkanen}. Let $A$ denote the standard annulus $\{z : 1 < |z| < 2\}$ of modulus 2. Consider the following collection of 16 rectangles of modulus 10:
\begin{alignat*}{2}
P_0 & = [1, 1.3] \times [-1.5,1.5], \qquad  && P_j = e(j/8) \cdot P_0, \\
Q_0 & = [0,2] \times [-0.1, 0.1], \qquad  && Q_j = e(j/8) \cdot Q_0,
\end{alignat*}
with $j=0,1,2,\dots, 7$. We claim that if $\gamma$ is a curve that  connects the boundary components of $A$,
then $\gamma$ contains a {\em horizontal} crossing of some $P_j$ or a {\em vertical} crossing of some $Q_j$. Indeed, if $\gamma$ is confined to
a sector $$\bigl \{z: \, -0.27 \pi + 0.25 \pi \cdot j \, < \, \arg z \, < \, 0.27 \pi + 0.25 \cdot j  \bigr \}$$ then $\gamma$ contains a horizontal crossing of $P_j$.
Howevever, since
$$
Q_j \subset \bigl \{z: \, -0.02 \pi + 0.25 \pi \cdot j \, < \, \arg z \, < \, 0.02 \pi + 0.25 \cdot j \bigr \},
$$
any wandering curve contains a vertical crossing of some $Q_j$.

An arbitrary annulus of modulus 2 in $U$ can be expressed as the image of $A$ under a complex-linear map $L(z) = az + b$. By assumption, for each $j = 0, 1,2,\dots, 7$, we can find a metric $\rho^*_{L(P_j)}$ of area $\le 10K$ which is admissible for $f \bigl (\Gamma_{\leftrightarrow}(L(P_j)) \bigr )$ and a metric $\rho^*_{L(Q_j)}$ of area $\le 10K$ which is admissible for $f \bigl (\Gamma_{\updownarrow}(L(Q_j)) \bigr )$. Since the metric
$$
\rho^*_{L(A)} = \sum_{j=0}^7 \rho^*_{L(P_j)} +  \sum_{j=0}^7 \rho^*_{L(Q_j)}
$$
has area at most $1960\, K$ and is admissible for the family of curves that connect the boundary components of $f(L(A))$, $\Mod f(L(A))$ is bounded from below by a definite constant (depending on $K$).

Finding an upper bound for $\Mod f(L(A))$ amounts to constructing a metric which is admissible for the family of curves which separate the boundary components of $f(L(A))$.
This only requires one rectangle, e.g.~$L(Q_0)$.
\end{proof}

\begin{remark}
Strangely enough, it is not known whether an orientation-preserving homeomorphism which distorts moduli of squares by a bounded amount must be  quasiconformal.
\end{remark}

Let  $\mathcal R(\varepsilon)$ denote the set of rectangles in the plane  whose sides have length at least $\varepsilon$. For a bounded domain $U \subset \mathbb{C}$, let $\mathcal R_U(\varepsilon)$ denote the collection of rectangles in $\mathcal R(\varepsilon)$ that are compactly contained in $U$. We say that an orientation-preserving homeomorphism $f: U \to V$ is $(K, \varepsilon)$ {\em roughly quasiconformal} if
 $$
(1/K) \cdot \Mod R \, \le \,  \Mod f(R) \, \le \, K \cdot \Mod R, \qquad R \in \mathcal R_U(\varepsilon),
 $$
 for some $K \ge 1$.
 We have the following compactness criterion for families of roughly quasiconformal maps:
 \begin{lemma}
 \label{compactness-criterion}
 Let $U \subset \mathbb{C}$ be a domain in the complex plane containing $0, 1$.
Suppose $f_n: U \to \mathbb{C}$ is a sequence of $(K, \varepsilon_n)$ roughly quasiconformal maps with $\varepsilon_n \to 0$ as $n \to \infty$.
Assume that $f_n(0)=0,\, f_n(1)=1$ for each $n$.
Then the sequence $\{ f_n \}$ is uniformly equicontinuous on compact subsets of $U$ and any subsequential limit is an $L(K)$ quasiconformal homeomorphism.
\end{lemma}

\begin{proof}
The proof of Lemma \ref{testing-rectangles} shows that $f_n$ distorts the moduli of all annuli $A(z, r, 2r) \subset U$ with $r \ge 5 \varepsilon_n$ by a bounded amount, depending on $K$. The proof of Lemma \ref{testing-annuli} shows that $f_n$ satisfies the doubling property at scales larger than $5  \varepsilon_n$, i.e.~there exist positive constants $\alpha < \beta$ depending only on $K$ such that
\begin{equation}
\label{eq:doubling}
0 \, < \, 2^{-\beta} \, \stackrel{(\ref{eq:quasisymmetry2})}{\le} \, \frac{\diam f_n \bigl (B(x,r) \bigr )}{\diam f_n \bigl (B(x,2r) \bigr)} \, \stackrel{(\ref{eq:quasisymmetry3})}{\le} \, 2^{-\alpha} \, < \, 1,
\end{equation}
if $r \ge 5  \varepsilon_n$ and $B(x, 4r) \subset U$.
Fix a compact connected set $F \subset U$ containing $0, 1$. Since $f_n(0) = 0$, $f_n(1) = 1$, iterated application of the doubling property shows that there exist constants $c_1(F), c_2(F) > 0$ such that
$$
B \bigl (f_n(x), c_1(F) r^{\beta} \bigr ) \, \subset \, f_n \bigl (B(x, r) \bigr ) \, \subset \, B \bigl (f_n(x), c_2(F) r^{\alpha} \bigr ),
$$
for any $x \in F$ and $5 \varepsilon_n \le r \le \frac{1}{4} \dist(F, \partial U)$.
It follows that any subsequential limit $f$ is a homeomorphism. Since $\Mod f_n \bigl (A(z,r,2r) \bigr ) \to \Mod f \bigl (A(z,r,2r) \bigr )$, $f$ verifies the assumptions of Lemma \ref{testing-annuli} and is therefore quasiconformal.
\end{proof}

Suppose $f: U \to \mathbb{C}$ is an orientation-preserving homeomorphism which is differentiable almost everywhere. By examining the behaviour of $f$ near points of differentiability, it is easy to see that $f$ is $K$ quasiconformal if and only if it distorts moduli of squares in $U$ by at most $K$. We will apply this observation when $f$ is the limit of roughly quasiconformal homeomorphisms $f_n$ where each $f_n$ distorts  moduli of squares $S \in \mathcal S_n$ by at most $K$, the collections $\mathcal S_n$ increase with $n$ and their union $\bigcup_{n = 1}^\infty \mathcal S_n$ is dense in the set of all squares in $U$ of diameter at most 1. For $\mathcal S_n$, we will use one of the collections below:

\begin{itemize}
\item
Let  $\mathcal S(\varepsilon)$ denote the set of squares in the plane with side length between $\varepsilon$ and 1. For a bounded domain $U \subset \mathbb{C}$, we let $\mathcal S_U(\varepsilon)$ denote
the collection of squares in $\mathcal S(\varepsilon)$ that are compactly contained in $U$.
\item
For $0 < \varepsilon \le 1$, let $\mathcal S'(\varepsilon)$ denote the set of squares in the plane that belong to one of the grids $e^{2\pi i (k/n)} \cdot j \varepsilon \mathbb{Z}^2$, $1 \le j, k \le n$  where $n = \lceil 1/\varepsilon \rceil$. If $U$ is a bounded domain, the collection $\mathcal S'_U(\varepsilon)$ of squares compactly contained in $U$ is finite.
By construction, $\bigcup_{\varepsilon > 0} \mathcal S'_U(\varepsilon)$ is dense in the set of all squares contained in $U$ of diameter at most 1.
\end{itemize}

\section{A lemma on percolation}
\label{sec:percolation}

In this section, we present a lemma on percolation which will allow us to control the moduli of  images of rectangles under an orientation-preserving homeomorphism that is quasiconformal off a random set of small measure.

Consider the infinite square grid $\mathbb{Z}^2$. Fix the percolation parameter $0 < r < 1$. In the discrete setting, we colour vertices of  $\mathbb{Z}^2$ in two colours: we colour a vertex {\em yellow} with probability $r$ and {\em blue} with probability $1-r$. For two points $x,y \in \mathbb{Z}^2$, we define
their {\em combinatorial distance} $d_{\mathbb{Z}^2}(x,y)$ as the minimal length of a path $$x_0, x_1, x_2, \dots, x_n, \qquad x_0 = x, \quad x_n = y,$$ where $x_i \sim x_{i+1}$ are adjacent vertices. By the Pythagorean theorem, the combinatorial
distance is within a factor of $\sqrt{2}$ of the Euclidean distance. We are more interested
in the {\em chemical distance} $d_{\chem}(x,y)$ which minimizes the number of blue vertices along paths that connect $x$ to $y$.
The following lemma says that if the points $x,y$ are at macroscopic distance from one another, then the chemical distance is also equivalent to the Euclidean distance:

\begin{lemma}
\label{mathieu-lemma}
There exists a universal constant $0 < r_0 <1/2$ so that if the percolation parameter $0 < r < r_0$ is sufficiently small, then with probability $\ge 1 - 1/N^2$, for two points $x,y \in [-N,N] \times [-N,N]$ with $|x-y| \ge \log N$,
$$
\frac{9}{10} \cdot d_{\mathbb{Z}^2}(x,y) \,\le\, d_{\chem}(x,y) \,\le\, d_{\mathbb{Z}^2}(x,y).
$$
\end{lemma}

\begin{proof}
The number of non-self intersecting  paths in $[-N,N] \times [-N,N]$ of length $L$ is bounded above by $(2N+1)^2 \cdot 4^L$ since there are  $(2N+1)^2$ choices for the initial vertex and at most four choices for each following vertex.
Since the probability that a fixed path of length $L$ contains at least $L/10$ yellow vertices is at most
$$\sum_{j=\lfloor L/10 \rfloor}^{L} {\binom{L}{j}} r^j(1-r)^{L-j} \, \le \, 2^L \sum_{j=\lfloor L/10 \rfloor}^{L} r^j \, \le \, 2^{L+1} \cdot r^{L/10},
$$
 the probability that some path of length $L \ge \log N$ contains at least $L/10$ yellow vertices is bounded above by
$$
2 (2N+1)^2 \cdot \sum_{L \ge \log N}  (8 r^{1/10})^L.
$$
A simple computation shows that the last quantity is at most $1/N^2$ provided $r$ is small enough.
\end{proof}

In the continuous setting, one independently colours each cell of $\mathbb{Z}^2$ either blue or yellow: {\em yellow} with probability $r$ and {\em blue} with probability $1-r$. The Euclidean distance minimizes the length of a rectifiable
curve that connects two given points in the plane. The continuous analogue of the chemical distance is defined by instead  minimizing the part of the length that is contained in the blue squares.

\begin{lemma}
\label{mathieu-lemma2}
There exists a universal constant $0 < r_0 < 1/2$ so that if the percolation parameter $0 < r < r_0$ is sufficiently small, then with probability $\ge 1 - 1/N^2$, for two points $x,y \in [-N,N] \times [-N,N]$ with $|x-y| \ge \log N$, their continuous chemical  distance is comparable to the Euclidean distance:
$$
\frac{1}{2} \cdot |x-y| \,\le\, |x-y|_{\chem} \,\le\, |x-y|.
$$
\end{lemma}

To deduce the continuous case from the discrete case, note that if the discrete itinerary of a path of length $L \ge \log N$ contains at most $L/10$ yellow squares, the total length that the continuous path spends in the blue squares is at least $L/2$.

\begin{modifications}
(i) By making $r_0 > 0$ sufficiently small, one can replace $9/10$ and $1/2$ in the lemmas above with arbitrary constants less than 1.

(ii) The proof works essentially verbatim for any graph whose vertices have bounded valence.

(iii)  Fix an odd  integer $m \ge 1$. As before, colour a cell {\em yellow} with probability $r$ and {\em blue} with probability $1-r$. We call a cell $\square$ {\em deep blue} if all cells $\hat \square \subset m \square$ are blue, where $m \square$ denotes the square with the
same center as $\square$ and  side length $m \cdot \ell(\square)$. We claim that there exists a universal constant $r_0(m)$ such that if the percolation
 parameter $r < r_0(m)$, then the total length that a continuous path in $\Gamma_{\leftrightarrow}$ spends in the deep blue squares is at least $N/2$.

To prove the discrete version of the claim, note that if there is a path  $x = x_0, x_1, x_2, \dots, x_n = y$ which contains more than $N/10$ vertices that are not deep blue, then there exists
a chain $x = x'_0, x'_1, x'_2, \dots, x'_n = y$ with $d_{\mathbb{Z}^2}(x'_i,x'_{i+1}) \le 2m-1$, $0 \le i \le n-1$ which contains more than $N/10$ yellow vertices. The continuous case follows from the discrete case.

(iv) We can slightly weaken the independence assumption when deciding to colour a cell blue or yellow: it is enough to require that the colours of any finite collection of cells $\square_1, \square_2, \dots, \square_n$ with
$d_{\mathbb{Z}^2}(\square_i, \square_j) \ge m$, $i \ne j$ are independent. As in (iii), $r_0$ depends on $m$. To see the discrete version of the statement, observe that if a path of length $L$ contains $\gtrsim L$ yellow vertices, then it contains an $m$-separated set of $\gtrsim L/m^2$ vertices. Again, the continuous case follows from the discrete case.
\end{modifications}

\section{Approximate conformality}

In this section, we run  percolation  with parameter $r$ on $S_1  = [-1, 1] \times [-1, 1]$ with mesh size $\delta = 1/N$.
The following lemma says that if an orientation-preserving homeomorphism is conformal outside a random set of small measure, then it is close to a conformal map:

\begin{lemma}
\label{blue-conformality}
For any $\varepsilon > 0$, there exist $r_0, N_0$ so that if $r < r_0$ and $N > N_0$, then with probability at least $1-\varepsilon$,
any orientation-preserving homeomorphism $f: S_1 \to \mathbb{C}$  that is conformal on the blue squares satisfies
$$
1- \varepsilon \, < \, \Mod f(S) \, < \, 1 + \varepsilon, \qquad \forall  S \in \mathcal S_{S_1}(\varepsilon).
$$
\end{lemma}

\begin{proof}
Recall that $\mathcal S_{S_1}(\varepsilon)$ denotes the collection of squares compactly contained in $S_1$ with $\ell(S) \ge \varepsilon$. To a square $S \in \mathcal S_{S_1}(\varepsilon)$, we associate the metric $\rho_S = \chi_{\mathscr B \cap S}$ where $\mathscr B$ is the union of the blue squares.
To estimate $\Mod f(S)$,  we use the conformal metric
  $$
  \rho_S^*(w) =  \bigl [ \Jac f^{-1}(w) \bigr ]^{1/2} \cdot \chi_{f(\mathscr B \cap S)}(w), \qquad w \in f(S).
  $$
By construction, $A(\rho_S^*) = A(\rho_S) \le \ell(S)^2$. By modification (i) of Lemma \ref{mathieu-lemma2}, for any $\eta > 0$, with probability at least
$1-\varepsilon$, the inequality
$\ell_{\rho^*_S}(f(\gamma)) \ge (1 - \eta) \ell(S)$ holds for any curve
$\gamma \in \Gamma_{\leftrightarrow}(S)$
and any square $S \in \mathcal S_{S_1}(\varepsilon)$, as long as $N > N_0$ is sufficiently large and $r < r_0$ is sufficiently small. In other words,
$\Mod f(S) \le (1-\eta)^{-2}$ for all $S \in \mathcal S_{S_1}(\varepsilon)$. It remains to choose $\eta = \varepsilon/3$.
\end{proof}

Let $\square \subset S_1$ be a cell which does not touch $\partial S_1$.
We say that an orientation-preserving homeomorphism $f: S_1 \to \mathbb{C}$ is {\em $\varepsilon$-close to linear} on $\square$ if
\begin{equation}
\label{eq:close-to-linear}
\|f - L_{\square} \|_{C(3\square)} \le \varepsilon \diam L(\square)
\end{equation}
 for some complex-linear map $L_{\square}(z) = a_{\square}z + b_{\square}$.

\begin{lemma}
\label{blue-conformality2}
For any $\varepsilon > 0$, there exist $r_0, N_0$ so that if $r < r_0$ and $N > N_0$, then with probability at least $1-\varepsilon$,
  any orientation-preserving homeomorphism $f: S_1 \to \mathbb{C}$ that is $\varepsilon$-close to linear on the blue squares satisfies
$$
1 - \varepsilon \, < \, \Mod f(S) \, < \, 1 + \varepsilon, \qquad \forall  S \in \mathcal S_{S_1}(\varepsilon).
$$
\end{lemma}

\begin{proof}
For an odd integer $m \ge 1$, let $\mathscr B_m$ denote the union of deep blue cells with depth parameter $m$ from  modification (iii) of  Lemma \ref{mathieu-lemma2}.
Given a square $S \in \mathcal S_{S_1}(\varepsilon)$ and an odd integer $m$,
 define $\rho_{S,m} = \chi_{\mathscr B_m \cap S}$ and
$$
\rho^*_{S,m} = \sum_{\square \subset \mathscr B_m} \frac{1}{a_{\square}} \cdot \chi_{f({\square})}.
$$
Then, $\rho_{S,m}$ is a metric on $S$ while $\rho^*_{S,m}$ is a metric on $f(S)$.
To a curve $\gamma \in \Gamma_{\leftrightarrow}(S)$, we associate a piecewise-linear curve $\overline \gamma = [z_1, z_2, z_3, \dots, z_n] \in \Gamma_{\leftrightarrow}(S)$ which has the same endpoints as $\gamma$ by straightening $\gamma$ in all cells $\square \subset S_1$.  To be precise, let us parameterize $\gamma: [0,1] \to \mathbb{C}$ in some way, e.g.~from left to right.
 Set $z_1 = \gamma(0)$.
We
examine the cell $\square_1$ that contains $z_1$, and replace the arc between $z_1$ and the last exit point of $\gamma$ in $\square_1$ by a line segment $[z_1,z_2]$. We continue replacing arcs of $\gamma$ by line segments, in the order of cells visited by $\gamma$, until we reach $z_n = \gamma(1)$.
 The astute reader may notice that if $\gamma$ exits a cell in a corner, then the next cell visited by $\gamma$ may not be uniquely defined -- in this case, we choose it arbitrarily. It is easy to see that the length of $\gamma$ can only decrease after straightening: $\ell_{\rho_{S,m+2}}(\overline{\gamma}) \le \ell_{\rho_{S, m}}(\gamma)$.

Since $f(\square) \subset (1+2\varepsilon)L(\square)$,
$
A(\rho^*_{S,1}) \le  (1+2\varepsilon)^2 \ell(S)^2  \le  (1+5\varepsilon) \ell(S)^2.
$
Inspection shows that for any curve $f(\gamma) \in \Gamma_{\leftrightarrow}(f(S))$, we have
\begin{equation}
\label{eq:errors-in-straightening}
\ell_{\rho^*_{S,1}} \bigl (f(\gamma) \bigr ) \, \ge \, (1-C_1\varepsilon) \cdot \ell_{ \rho^*_{S,3}} \bigl (\overline{f(\gamma)} \bigr ) \, \ge \, (1-C_2\varepsilon) \cdot \ell_{\rho_{S,5}}(\overline{\gamma}),
\end{equation}
where $\overline{f(\gamma)} = [f(z_1), f(z_2), \dots, f(z_n)]$. The first error in (\ref{eq:errors-in-straightening}) is due to the fact that  $a_{ \square} / a_{\hat \square} \le 1 + 2\varepsilon$ for two adjacent deep blue cells $\square$ and $\hat \square$, which makes line segments approximate geodesics rather than
genuine geodesics, while the second error  in (\ref{eq:errors-in-straightening}) is due to the discrepancy between $f$ and $L$.

By modifications (i) and (iii) of Lemma \ref{mathieu-lemma2}, (\ref{eq:errors-in-straightening}) implies that when $N > N_0$ is large and $r < r_0$ is small, with probability $\ge 1-\varepsilon$,
$$\ell_{\rho_{S,1}^*}(f(\gamma)) \ge (1- C_3\varepsilon) \ell(S)$$
for any curve $f(\gamma) \in \Gamma_{\leftrightarrow}(f(S))$ and any $S \in \mathcal S_{S_1}(\varepsilon)$. In other words, $\Mod f(S) \le 1 + C \varepsilon$ for all $S \in \mathcal S_{S_1}(\varepsilon)$ where
$C > 0$ is a universal constant.
\end{proof}

\section{Homogenization of quasiconformal maps}
\label{sec:homogenization-qc}

In this section, we prove a variant of Theorem \ref{main-thm} where the dilatation is randomized on a bounded open set $\Omega \subset \mathbb{C}$.
\begin{theorem}
\label{main-thm2}
Let  $\Omega \subset B(0,R)$ be a bounded open set in the plane whose boundary has zero measure. Consider the square grid in the plane of mesh size $\delta > 0$. For each square of side length $\delta$ compactly contained in $\Omega$, select $\mu$ according to the measure $\lambda$.
Outside of the $\delta$-approximation of $\Omega$, set $\mu = \mu_0$, where $\mu_0$ is a fixed Beltrami coefficient on the plane with $\|\mu_0\|_\infty < 1$.  Then,
$$
\mathbb{P} \biggl [ \, \sup_{z \in B(0,R)} \Bigl |  w^{\mu}(z) -  w^{\mu_{\lambda} \cdot \chi_\Omega + \mu_0 \cdot \chi_{\mathbb{C} \setminus \Omega}}(z)  \Bigr | < \varepsilon \, \biggr ] \to 1, \qquad \text{as }\delta \to 0,
$$
where $\mu_{\lambda}$ is a constant that depends only on $\lambda$.
\end{theorem}

By the results of Section \ref{sec:percolation}, random quasiconformal mappings are roughly quasiconformal:

\begin{lemma}
\label{random-maps-are-roughly-qc}
For any $\varepsilon > 0$, when the mesh size $\delta < \delta_0(\varepsilon)$ is sufficiently small, the probability that
$ w^{\mu}$ is $(K, \varepsilon)$ roughly quasiconformal on $B(0,2R)$ is at least $1-\varepsilon$.
\end{lemma}

  \begin{proof}
  Choose $0 < k_1 < 1$ so that
  $$
  \| \mu_0 \|_\infty \le k_1, \qquad
   \lambda \bigl (\{z : k_1 < |z| < 1\} \bigr )< r_0,
  $$
   where $r_0$ is the constant from Lemma \ref{mathieu-lemma2}. Colour a cell $\square$ in $\delta \mathbb{Z}^2$ {\em yellow} if $|\mu(\square)| >k_1$ and {\em blue} otherwise. Let $\mathscr B$ denote the union of the blue squares.  We will show the lemma holds with $K = 4K_1$ where  $K_1 = \frac{1+k_1}{1-k_1}$.

Recall that $\mathcal R_{B(0,2R)}(\varepsilon)$ denotes the collection of rectangles  compactly contained in $B(0,2R)$ whose sides have length at least $\varepsilon$.
 To a rectangle $R \in \mathcal R_{B(0,2R)}(\varepsilon)$, associate the metric $\rho_R = \chi_{\mathscr B \cap R}$. By Lemma \ref{mathieu-lemma2}, when the mesh size $\delta < \delta_0(\varepsilon)$ is sufficiently small, with probability at least $1-\varepsilon$, $\ell_{\rho_R}(\gamma) \ge \ell_2(R)/2$ for any curve $\gamma \in \Gamma_{\leftrightarrow}(R)$ and any square $R \in \mathcal R_{B(0,2R)}(\varepsilon)$.

 To estimate $\Mod w^{\mu}(R)$, we use the metric
  $$\rho_R^*(w) =  \bigl [ \Jac (w^{\mu})^{-1}(w) \bigr ]^{1/2} \cdot \chi_{w^{\mu}(\mathscr B \cap R)}(w), \qquad w \in w^{\mu}(R).
  $$
 By construction, $A(\rho_R^*) = A(\rho_R) \le \ell_1(R) \ell_2(R)$. Since $w^{\mu}$ is $K_1$-quasiconformal on the support of $\rho_R$, $$\ell_{\rho_R^*}(w^{\mu}(\gamma)) \ge (1/\sqrt{K_1}) \cdot \ell_{\rho_R}(\gamma), \qquad \gamma \in \Gamma_{\leftrightarrow}(R).$$
It follows that $\ell_{\rho_R^*}(w^{\mu}(\gamma)) \ge \ell_2(R)/(2\sqrt{K_1})$  and $\Mod w^{\mu}(R) \le 4K_1 \cdot \Mod R$. The proof is complete.
  \end{proof}

We say that two squares $S_1, S_2$ have the same {\em orientation} if $S_2 = a S_1 + b$ where $a > 0$ and $b \in \mathbb{C}$.
To motivate our proof of Theorem \ref{main-thm2}, note that a complex-linear mapping of the plane preserves moduli of rectangles, while an affine mapping $A$ that is not complex-linear has an extremal direction: the modulus of any geometric rectangle oriented in this direction is stretched by $K(A)$, while the moduli of any other rectangle is stretched by a strictly  smaller amount.
 More generally,  a quasiconformal mapping $\varphi$ has constant dilatation $\mu_\varphi = \mu_A$ on $\Omega$ if and only if $\varphi$ stretches moduli of rectangles pointing in the $A$ direction by  $K(A)$.
In order to show that a random quasiconformal mapping behaves like $A$ on $\Omega$, we need a mechanism for identifying this direction.

For a square $S \subset \Omega$, look at the probability that $\Mod  w^{\mu}(S) > K$ and take $\limsup$ as $\delta \to 0$. Let $K^*(S)$ be the infimum of $K > 0$ for which this $\limsup$ is 0. Thus $K^*(S)$ measures the maximal effective distortion of $S$.  Maximizing $K^*(S)$ over all squares $S \subset \Omega$, we obtain the constant $K^*$ which measures the maximal effective distortion of the model. Since the product of the moduli of the ``horizontal'' and ``vertical'' families of a conformal rectangle is 1, $K^* \ge 1$.
Since with high probability, a random quasiconformal mapping is roughly quasiconformal, $K^*$ is finite.

\begin{lemma}
{\em (i)} $K^*(S)$ depends only on the orientation of $S$ and not its side length or location within $\Omega$, or on the domain $\Omega \subset \mathbb{C}$.

{\em (ii)} For a square $S$, let $S_\theta$ denotes the square obtained by rotating $S$ by $e^{i\theta}$ around its center.
Let $S$ be a square for which all $S_{\theta}$, $\theta \in [0,2\pi]$ are in $\Omega$.
The function $\theta \to K^*(S_\theta)$ is continuous.

{\em (iii)} There exists $\theta^* \in [0,2\pi]$ such that for any $\varepsilon > 0$, there exists a constant $c(\varepsilon) > 0$ and a sequence of scales
$\delta_{\varepsilon, j} \to 0$ for which
\begin{equation}
\label{eq:special-scale}
\mathbb{P}_{\delta_{\varepsilon, j}} \bigl ( \Mod  w^{\mu}(S_{\theta^*}) > K^* - \varepsilon \bigr ) > c(\varepsilon).
\end{equation}
\end{lemma}

\begin{proof}[Sketch of proof.] If $S, S' \subset \Omega$ are two squares with the same orientation, then the distribution of the random variable $\Mod  w_{\delta}^{\mu}(S)$ is essentially the same as that of
$\Mod  w_{\delta'}^{\mu}(S')$ with $\delta' = \delta \cdot \ell(S')/\ell(S)$, i.e.~
$$
\Delta(t) \, = \,  \mathbb{P} \bigl (\Mod  w_{\delta}^{\mu}(S) < t \bigr ) -  \mathbb{P} \bigl (\Mod  w_{\delta'}^{\mu}(S') < t \bigr ),
$$
tends weakly to 0 as $\delta \to 0$. The reason that $\Delta(t)$ could be non-zero comes from the slight discrepancy of how the $\delta$ and $\delta'$ grids intersect $S$ and $S'$, however, by rough quasiconformality and Lemma \ref{continuity-of-moduli}, this discrepancy is essentially negligible if the grids are very fine.
This proves (i). The same circle of ideas also show (ii) and (iii).
\end{proof}

Let $A_{\lambda}$ be the affine transformation with dilatation $K^*$ which stretches all squares pointing in the $S_{\theta^*}$ direction by $K^*$ and fixes the points $0,1$. We denote the dilatation of $A_{\lambda}$ by $\mu_{\lambda}$. The task before us is now clear: we want to show that $w^{\mu}$ is close to the normalized quasiconformal map $\Phi$ with  dilatation $\mu_\lambda$ on $\Omega$ and $\mu_0$ on $\mathbb{C} \setminus \Omega$
in $C \bigl (B(0,R) \bigr )$. Note that $\Phi$ is uniquely determined since $\partial \Omega$ has Lebesgue measure 0. Alternatively, we can show that
$f := A_\lambda^{-1} \circ w^{\mu}$ is close to $A_\lambda^{-1} \circ \Phi$. We do this in a series of incremental improvements. The first step is to promote positive probability to high probability:

 \begin{lemma}
 \label{extremal-direction1}
There is a sequence of scales $\delta_j \to 0$ such that  \begin{equation}
 \label{eq:good-scale}
 \mathbb{P}_{\delta_j} \bigl ( \Mod  w^{\mu}(S_{\theta^*}) > K^* - 1/j \bigr ) > 1 - 1/j.
 \end{equation}
\end{lemma}

The proof rests on the following lemma:

\begin{lemma}
\label{deficiency-lemma}
Suppose $S$ is a square in the plane and $\varphi: S \to \mathbb{C}$ is a $K$ quasiconformal map. For an integer $n \ge 1$, divide $S = S_1 \cup S_2 \cup \dots \cup S_{n^2}$ into $n^2$ squares of equal size. If $\Mod \varphi(S_i) \le K_0$ for at least $c \cdot n^2$ of these squares, then $\Mod \varphi(S) \le K_1$ for some  constant $K_0 < K_1 < K$ which depends on $K, K_0, c$ but not on $n$.
\end{lemma}

\begin{proof}
Since any path in $\Gamma_{\leftrightarrow}(S)$ travels $\ge \ell(S)$ horizontally, the metric
$$
\rho^*(w) = \frac{1}{\ell(S)} \cdot \biggl | \frac{\partial \varphi}{\partial x} (\varphi^{-1}(w)) \biggr |^{-1}, \qquad w \in \varphi(S),
$$
is admissible for $\Gamma_{\leftrightarrow}(\varphi(S))$. Since
\begin{equation}
\label{eq:deficiency}
\biggl | \frac{\partial \varphi}{\partial x} \biggr |^2 \ge (1/K) \cdot \Jac \varphi
\quad \implies \quad
\biggl | \frac{\partial \varphi}{\partial x} \circ \varphi^{-1} \biggr |^{-2} \le K \cdot \Jac \varphi^{-1},
\end{equation}
the area $A(\rho^*) \le K$, which shows that $\Mod \varphi(S) \le K$ and equality holds if and only if $\varphi$ is the extremal stretch by $K$ in the horizontal direction. A compactness argument shows that  (\ref{eq:deficiency}) has a definite defect on any good square $S_i$, $A(\rho^* \cdot \chi_{\varphi(S_i)}) \le \frac{(1-\varepsilon)K}{n^2}$, which gives the required improvement.
\end{proof}

 \begin{proof}[Proof of Lemma \ref{extremal-direction1}]
If the lemma were false, there would exist constants $K_0 < K^*$ and $c_0 > 0$ such  that
 \begin{equation}
 \label{eq:contradiction}
 \mathbb{P}_{\delta} \bigl ( \Mod w^{\mu}(S_{\theta^*}) < K_0 \bigr ) \ge c_0, \qquad \text{ for any }\delta > 0 \text{ sufficiently small}.
\end{equation}
Assuming this, we will construct a sequence of quasiconformal maps $\varphi_k$ with the following properties:
\begin{enumerate}
\item[(i)] $\varphi_k$ is $(K, 1/k)$ roughly quasiconformal on $B(0,2R)$ where $K$ is from Lemma \ref{random-maps-are-roughly-qc}.
\item[(ii)] $\Mod \varphi_k(\sigma) \le K^* + 1/k$ for all $\sigma \in \mathcal S'_{\Omega}(1/k)$, where $\mathcal S'_{\Omega}(1/k)$ is a finite collection of squares in $\Omega$ of side length $\ge 1/k$ which was defined
in Section \ref{sec:roughly-qc}.
\item[(iii)]
$\Mod \varphi_k(S_{\theta^*}) > K^* - \varepsilon > K_1 > K_0$, where $K_1$ is given by Lemma \ref{deficiency-lemma} with $K = K^*$, $c = c_0/8$ and $\varepsilon > 0$ is any constant less than $K^* - K_1$.
\item[(iv)]
For an integer $n \ge 1$, divide $S_{\theta^*} = S_1 \cup S_2 \cup \dots \cup S_{n^2}$ into $n^2$ squares of equal size, where $n$ is a  positive integer that will be chosen below.
For  at least $(c_0/8) n^2$ of these squares, $\Mod \varphi_k(S_i) \le K_0$.
\end{enumerate}

By (i), the sequence of mappings $\varphi_k$ is precompact, (ii) implies that any subsequential limit  $\varphi$
is $K^*$ quasiconformal, (iii) tells us that $\Mod \varphi(S_{\theta^*}) \ge K^* - \varepsilon$, while (iv)
ensures that $\Mod \varphi_k(S_i) < K_0$ for at least $(c_0/8)n^2$ of the small squares $S_i$. However, these properties are incompatible by Lemma \ref{deficiency-lemma}.

Clearly, (iii) holds with probability $\ge c(\varepsilon)$  for the special scales $\delta_{\varepsilon, j}$ from (\ref{eq:special-scale}).
 According to Lemma \ref{random-maps-are-roughly-qc}, by requesting the mesh size $\delta_{\varepsilon, j(k)}$ to be small,  we can ensure that the probability that $\varphi_k$ is $(K, 1/k)$ quasiconformal exceeds $1 - c(\varepsilon)/4$.
Since the number of squares in $\mathcal S'_{\Omega}(1/k)$ is finite, if
 $\delta_{\varepsilon, j(k)}$ is small,  then
\begin{equation}
\label{eq:background-squares}
\mathbb{P}\Bigl (\Mod  w^{\mu}(\sigma) \le K^* + 1/k, \, \forall \sigma \in \mathcal S'_{\Omega}(1/k) \Bigr ) \, > \, 1- c(\varepsilon)/4.
\end{equation}

Discard $\sim 3n^2/4$ of the small squares $S_i$ so that the remaining squares are a definite distance apart (and therefore the  moduli of their images are independent). By
(\ref{eq:contradiction}) and the law of large numbers, we can pick $n$ sufficiently large so that for arbitrarily small $\delta > 0$, with probability at least $1-c(\varepsilon)/4$, $\Mod \varphi(S_i) \le K_0$
for at least $(c_0/8)n^2$  squares $S_i$.

To summarize, if we choose $\delta_{\varepsilon, j(k)}$ in a suitable manner, then with positive probability, the random quasiconformal map $\varphi_k = w^{\mu}$ satisfies (i)--(iv).
\end{proof}

A compactness argument similar to the one in Lemma \ref{extremal-direction1} shows:

 \begin{lemma}
 \label{extremal-direction2}
There exists a sequence of scales $\delta'_j \to 0$ so that if $S \subset \Omega$    has the same orientation as $S_{\theta^*}$,
then with probability $> 1 - 1/j$,

{\em (i)} $f$ is $(K  K^*, \ell(S)/j)$ roughly quasiconformal on $2S$,

{\em (ii)} $\Mod f(\sigma) < 1 + 1/j$  for all $\sigma \in \mathcal S'_{S}(\ell(S)/j)$,

{\em (iii)} $\Mod f(S) > 1 - 1/j$,

\noindent when the mesh size $\delta = \ell(S)/\ell(S_{\theta^*}) \cdot \delta'_j$  is  adapted to $S$.
\end{lemma}

Note that the conditions in Lemma \ref{extremal-direction2} are local in nature: they only depend on the behaviour of $\mu$ in $2S$.
When they are satisfied, $f$ is close to conformal on $S$, and therefore, by Koebe's distortion theorem, $f$ is close to linear deep inside $S$:

\begin{lemma}
\label{extremal-direction3}
For any $0 < \eta < 1/4$, there exists a $j(\eta)$ so that if $(f, S)$ satisfies {\em (i)--(iii)} from Lemma \ref{extremal-direction2}, then

{\em (iv)} $f$ is within $\eta^2 \cdot \diam f(S)$ of a conformal map $\varphi: S \to \mathbb{C}$,

{\em (v)} $f$ is $C\eta $-close to linear on $\eta S$  where $C = C(KK^*) > 0$,

\noindent when the mesh size $\delta = \ell(S)/\ell(S_{\theta^*}) \cdot \delta'_{j(\eta)}$  is adapted to $S$.
\end{lemma}

Property (iv) follows from a compactness argument, while (v) follows from Koebe's distortion theorem and the doubling property (\ref{eq:doubling}) which tells us that
$\diam f(\frac{1}{2} S) \asymp \diam f(S)$.
We now eliminate the need to use a subsequence of scales:

 \begin{lemma}
 \label{extremal-direction}
For any square $S \subset \Omega$ and $\varepsilon > 0$, when the mesh size $\delta < \delta_0(\varepsilon, S)$ is sufficiently small, $$
 \mathbb{P}_{\delta} \bigl ( 1 - \varepsilon < \Mod  f(S) < 1 + \varepsilon \bigr ) > 1 - \varepsilon.
$$
\end{lemma}

\begin{proof}
Suppose $0 < \delta < \delta'$ where $\delta' = \delta'_{j(\eta)}$.  Consider the grid $e^{i\theta^*} \beta\mathbb{Z}^2$ which consists of squares of
side length $\beta = \eta \cdot \delta/\delta' \cdot  \ell(S_{\theta^*})$ that have the same orientation as $S_{\theta^*}$.
We colour a cell $\square \in e^{i\theta^*} \beta\mathbb{Z}^2$ {\em blue} if $(f, \eta^{-1}  \square)$ satisfies conditions (i)--(iii) of Lemma \ref{extremal-direction2} and {\em yellow} otherwise.
According to Lemma \ref{extremal-direction2}, the probability that any given cell is blue is at least $1 - 1/j$.
 Even though the colours of the cells are not independent, the colour of a cell only depends on the behaviour of the Beltrami coefficient $\mu$ in $2\eta^{-1} \, \square$.

By Lemma \ref{extremal-direction3}, $f$ is $C\eta$-close to linear on the blue cells. In view of modification (iv) of Lemma \ref{mathieu-lemma2}, Lemma \ref{blue-conformality2} tells us that with high probability, $\Mod f(S)$ is close to 1: the error can be made arbitrarily small by requesting $\eta$ to be small and $j(\eta)$ to be large.
\end{proof}

It is now a simple matter to prove Theorem \ref{main-thm2}:

\begin{proof}[Proof of Theorem \ref{main-thm2}.]
As noted previously, to show that $w^{\mu}$ is close to $\Phi$, we can instead show that
$f = A_\lambda^{-1} \circ w^{\mu}$ is close to $A_\lambda^{-1} \circ \Phi$.
By Lemmas  \ref{random-maps-are-roughly-qc}
and  \ref{extremal-direction}, for any $\varepsilon > 0$, if the mesh size $\delta < \delta_0(\varepsilon)$ is small, then with  probability $\ge 1 - \varepsilon$,
\begin{equation}
\label{eq:condition1}
f \text{ is }(KK^*, \varepsilon) \text{ roughly quasiconformal on }B(0,2R),
\end{equation}
\begin{equation}
\label{eq:condition2}
 1 - \varepsilon \, \le \, \Mod f(\sigma) \, \le \, 1 + \varepsilon, \qquad \forall \sigma \in \mathcal S'_{\Omega}(\varepsilon).
\end{equation}
\begin{equation}
\label{eq:condition3}
\frac{\overline{\partial }f}{\partial f}\, = \, \frac{\overline{\partial} (A^{-1}_\lambda \circ \Phi)}{\partial (A^{-1}_\lambda \circ \Phi)}, \qquad \text{on }\mathbb{C} \setminus \Omega.
\end{equation}
If the theorem were false, we would have a sequence of quasiconformal mappings $\varphi_n$ which satisfy the above conditions with $\varepsilon = 1/n$, but were a definite distance away from $A^{-1}_\lambda \circ \Phi$ in $C \bigl (B(0,R) \bigr )$.
This is impossible since $A^{-1}_\lambda \circ \Phi$ is the only possible limit of such a sequence.
\end{proof}

\section{Random q.c.~mappings on the plane}

Let $\mu$ be a random Beltrami coefficient on the plane, constructed with help of the probability measure $\lambda$. Since $K_\mu = \frac{1 + |\mu|}{1-|\mu|}$ is not bounded in general, one may wonder if there is a homeomorphism with dilatation $\mu$.
Uniqueness follows from the fact that any two homeomorphisms of the plane with the same dilatation differ by a conformal automorphism. For existence, we only need $K_\mu$ to be locally bounded,
although a priori, the homeomorphism may not be surjective. Namely, for any $R > 2$, we can truncate $\mu_R =\mu \cdot \chi_{B(0,R)}$ and use the measurable Riemann mapping theorem to construct a quasiconformal map
with dilatation $\mu_R$ that fixes $-1,0,1$. A normal families argument shows that there is a homeomorphism of the Riemann sphere with dilatation $\mu$. One can then post-compose it with a M\"obius transformation
to make it fix $0, 1, \infty$.

We now check that $w^{\mu}$ is surjective almost surely.
Choose $0 < k < 1$ so that $\lambda \bigl (\{z : k < |z| < 1\} \bigr ) < r_0$ where $r_0$ is from Lemma \ref{mathieu-lemma2}.
Colour a cell $\square$ in $\delta \mathbb{Z}^2$ {\em yellow} if
$|\mu(\square)| >k$ and {\em blue} otherwise.
 From the construction of $\mu$, the probability that a cell is yellow is less than $r_0$.
  Since $\sum_{N=1}^\infty (\delta/N)^2 < \infty$,  the Borel-Cantelli lemma and Lemma \ref{mathieu-lemma2} show that almost surely, for all sufficiently large $N \ge N_0(\mu)$, the homeomorphism $ w^{\mu}(z)$ is $(K, \delta \log(N/\delta))$ roughly quasiconformal on $[-N,N] \times [-N,N]$.
This makes the moduli of every sufficiently large annulus $w^{ \mu}\bigl (A(0, N,2N) \bigr)$ bounded from below which forces $w^{\mu}$ to be surjective.

Fix a ball $B(0,R)$ with $R > 2$. We now show that when $\delta > 0$ is sufficiently small, then with
 with probability at least $1 - \varepsilon$, $w^{\mu}(z)$ is within $\varepsilon$ of the affine mapping $A_{\lambda}(z)$ on $B(0,R)$.
For any $\gamma > 1$,
 the compactness of roughly quasiconformal mappings tell us that when the mesh size $\delta < \delta_0(R, \varepsilon, \gamma)$ is sufficiently small,
 $ w^{\mu}|_{B(0, \gamma R)}$ is close to a quasiconformal map $\Phi^\mu: B(0, \gamma R) \to \mathbb{C}$ with constant dilatation $\mu_\lambda$. By requesting $\gamma$ to be large and applying
  Koebe's distortion theorem, we see that $(A_\lambda^{-1} \circ  \Phi^{\mu})|_{B(0, R)}$ is close to a linear map. Since $A_\lambda^{-1} \circ  \Phi^{\mu}$ fixes the points 0 and 1, it is essentially the identity.

\section{Moduli of rectangles in circle packings}
\label{sec:moduli-CP}

For a combinatorial rectangle $R$ in a circle packing, we have two different notions of moduli for curves connecting the opposite sides of $R$: the discrete modulus of the triangulation and the continuous modulus of the carrier. In general, the two notions of modulus are unrelated, however, when $\mathcal P$ has bounded geometry,
the continuous modulus and the discrete modulus agree up to a multiplicative constant.

The Euclidean Ring Lemma \cite[Lemma 8.2]{stephenson-book} says that for any $N \ge 3$, there exists a constant $\mathfrak c(N) > 0$ such that if $C = B(v, r)$ is an interior circle in $\mathcal P$ whose degree is at most $N$, then $r_i \ge \mathfrak c(N) \cdot r$ for any circle $C_i = B(v_i, r_i) \in \mathcal P$ tangent to $C$. In this case,  $B \bigl (v, (1+\mathfrak c(N)/2) r \bigr )$ can intersect at most $N+1$ circles from $\mathcal P$.

\begin{lemma}
\label{moduli-CP}
Let $R$ be a combinatorial rectangle in a circle packing $\mathcal P$. Suppose $\rho_{\disc}$ is a discrete metric defined on the vertices of the underlying triangulation, which
 is supported on the set of vertices of degree at most $N$. Let $\eta = \mathfrak c(N)/2$.
Define a continuous metric by the formula
\begin{equation}
\label{eq:cont-metric-def}
\rho_{\cont}(z) \, := \, \frac{1}{\eta} \sum_{B(v_i, r_i) \in \mathcal P}  \frac{ \rho_{\disc}(v_i) } {r_i} \cdot \chi_{B \bigl (v_i, (1+\eta)r_i \bigr )}(z).
\end{equation}
If $\rho_d$ was admissible for the horizontal curve family $\Gamma_{\leftrightarrow}^{\disc}$ in the discrete sense, then $\rho_c$ will be admissible for  $\Gamma_{\leftrightarrow}^{\cont}$ in the continuous sense.
Furthermore, the total area of $\rho_{\cont}$ is controlled by the total area of $\rho_{\disc}$: $A(\rho_{\cont}) \le C(N) A(\rho_{\disc})$.
\end{lemma}

\begin{proof}
The bound on the total area is clear since the sum defining $\rho_{\cont}(z)$ has at most $N+1$ non-zero terms.
We now check that $\ell_{\rho_{\cont}} (\gamma) \ge 1$ for $\gamma \in \Gamma^{\cont}_{\leftrightarrow}$. Observe that since the sides of every triangle in $\mathcal P$ are contained in $\bigcup \overline{C_i}$, the combinatorial progress that $\gamma$ makes through the triangulation is recorded by the collection of circles it visits. However, the $\rho_c$ cost of entering (or exiting) the influence of a circle is at least $\rho_{\disc}(v_i)$.
\end{proof}

\begin{remark}
To study the boundary behaviour of maximal circle packings, we need to allow $R$ to be an extended combinatorial rectangle. For a boundary circle $C_i = B(v_i, r_i)$ in $\mathcal P$, we let $v^*_i$ denote the point of tangency between $C_i$ and the unit circle. We extend the underlying triangulation of $\mathcal P$ by adding  edges from  $v_i$ to $v_i^*$ and from $v_i^*$ and $v_j^*$ if  $C_{v_i}$ and $C_{v_j}$ are tangent. With this definition, the extended ``triangulation''  also includes some quadrilaterals. By an {\em extended combinatorial rectangle}, we simply mean a combinatorial rectangle in the extended triangulation.
It is easy to see that if $\mathcal P$ is a maximal circle packing, the Euclidean Ring Lemma also applies to boundary circles. We leave it to the reader to check that Lemma \ref{moduli-CP} also holds for extended combinatorial rectangles.
\end{remark}

\section{Random Delauney triangulations}
\label{sec:random-delauney}

In this section, we prove Theorem \ref{main-thm4} which says that a circle packing of a random Delauney triangulations approximates a conformal map. For technical reasons, it is preferable to use a slightly different construction of a random Delauney triangulation where the Delauney points are chosen according to a Poisson point process of high intensity. We recall the definition. For a measurable set $E \subset \Omega$, we denote its Euclidean area by $A(E)$ and the number of Poisson points
contained in $E$ by $N_E$.

 A {\em Poisson point process} of intensity $\lambda$ produces a random collection of points in $\Omega$ according to the following two axioms:

(1) For any measurable set $E \subset \Omega$, $$\mathbb{P}(N_E = n) = e^{-A(E)\lambda} \cdot \frac{(A(E) \lambda)^n}{n!}.$$

(2) If $E_1, E_2, \dots, E_k$ are disjoint measurable sets, then $N_{E_1}, N_{E_2}, \dots, N_{E_k}$ are independent random variables.

From the uniqueness of the Poisson point process, it follows that the union of two Poisson point processes is also a Poisson point process and the intensities add. The law of large numbers tells us that when the intensity
$\lambda$ is large, then with high probability (w.h.p.) $N_E \sim A(\Omega) \lambda$.

We will need the following  estimate:

\begin{lemma}
\label{points-in-a-square}
 Suppose $ \Omega \subset \mathbb{C}$ is a bounded domain. For any $0 < \varepsilon < 1$, when $\lambda \ge \lambda_0(\varepsilon, \Omega)$ is sufficiently large, with probability $1-\varepsilon$, the estimate
 $$
(1-\varepsilon) \cdot A(R) \lambda \, \le \, N_R \, \le \, (1+\varepsilon) \cdot A(R) \lambda
$$
holds for every rectangle $R \subset \Omega$ whose sides have length at least $\varepsilon$.
\end{lemma}

\begin{proof}
 Take $\delta = (1/20) \varepsilon^2$ and consider all cells in the square grid $\delta \mathbb{Z}^2$ which intersect $\Omega$.
The lemma follows from the following two observations:

(i) By the law of large numbers, when $\lambda$ is large, w.h.p.~
$$
(1-\varepsilon/4) \cdot A(\square)\lambda \, \le \, N_{\square} \, \le \, (1+\varepsilon/4) \cdot A(\square)\lambda$$
for any cell $\square \in \delta \mathbb{Z}^2$ that is completely contained in $\Omega$, while the upper bound holds for any $\square  \in \delta \mathbb{Z}^2$ that merely intersects $\Omega$.

(ii) Given a rectangle $R \subset \Omega$ whose sides have length at least $\varepsilon$, let
$E_1$ be the union of cells in $\delta \mathbb{Z}^2$ that are completely contained in $R$, and $E_2$ be the union of cells in $\delta \mathbb{Z}^2$ that have non-empty intersection with $R$. The areas of $R \setminus E_1$ and $E_2 \setminus R$ are bounded above by  $(\varepsilon/4) \cdot A(R)$.
\end{proof}

One can deduce Theorem \ref{main-thm4} for the original model where the number of Delauney points is fixed by using the following simple observation: for any $\varepsilon > 0$, when $N$ is large, w.h.p.~a collection of $N$ random points is squeezed between Poisson point processes with intensities $N/A(\Omega) - \varepsilon$ and
$N/A(\Omega) + \varepsilon$.

\subsection{Basic properties of Delauney triangulations}

Below, we will repeatedly use the following property of Delauney triangulations:

\begin{itemize}
\item[$(*)$] If $x$ belongs to a Delauney edge $v_1v_2$, then the closest point to $x$ is either $v_1$ or $v_2$.
\end{itemize}

The following lemma says that when the intensity is large, Delauney triangulations tend to have short edges and exhaust $\Omega$:

\begin{lemma}
\label{basic-properties}
 Suppose $ \Omega \subset \mathbb{C}$ is a Jordan domain with $C^1$ boundary. For any $\varepsilon > 0$ and compact set $K \subset \Omega$, when $\lambda \ge \lambda_0(\varepsilon, K, \Omega)$ is sufficiently large, with probability $1-\varepsilon$, we have:

{\em (i)} The length of any edge of $\mathcal T$ is less than $\varepsilon$,

 {\em (ii)}  $\carr \mathcal T \supset K$.

\end{lemma}

\begin{proof}
For $\delta > 0$, let $\Omega^{\delta}$ be the union of all cells in $\delta \mathbb{Z}^2$ contained in $\Omega$. To define a partition of $\Omega$,  distribute the mass $\Omega \setminus \Omega^{\delta}$ amongst the boundary cells of $\Omega^{\delta}$.
Since $\partial \Omega$ is $C^1$, for small $\delta$, we can distribute the excess mass so that the diameter of every cell is at most $3\delta$ (the number 3 could be replaced by any other constant  greater than $\sqrt{5}$). In this case, the area of any cell would be comparable to $\delta^2$.
When the intensity $\lambda > 0$ is large, with probability $1-\varepsilon$, every
cell in the above partition contains at least one point of $\mathcal T$.

(i) Suppose $v_1 \sim v_2$ are neighbouring vertices of $\mathcal T$, one of which is at least $10\delta$ away from the boundary. If the Delauney edge $v_1v_2$ is longer than $3\sqrt 2 \delta$, then its midpoint $x$ would belong to a cell that contains a Delauney vertex which is not $v_1$ nor $v_2$. This would contradict property $(*)$ above. The same argument also works when $v_1$ and $v_2$ are close to the boundary, but one would need to use a larger constant, e.g.~$9 \delta$ works.

(ii) Suppose $v \in \mathcal T$ is a Delauney point such that $\dist(v, \partial \Omega) > 5 \delta$. We claim that the Vorononi cell $F_v \subset \overline{B(v, \sqrt{2}\delta)}$. To see this, note that if $y \in \partial B(v, \sqrt{2}\delta)$, then $y$ cannot lie in the same cell of the grid $\delta \mathbb{Z}^2$ as $v$. By assumption, there is a Delauney point $v' \ne v$ in the cell that contains $y$. Since cells of $\delta \mathbb{Z}^2$ have diameter $\sqrt{2}\delta$, $|v' - y| \le |v - y|$, which shows that $y$ is either outside $F_v$ or on the boundary of $F_v$. This proves the claim.
Let $v_1, v_2, v_3, \dots, v_d \in \mathcal T$  be the vertices connected to $v$ by an edge, listed in counter-clockwise order. Since the midpoint of the edge $vv_i$ lies on $F_v \cap F_{v_i}$, $|v - v_i| \le 2\sqrt{2}\delta$.
It follows that the polygon $v_1v_2 \dots v_d$ is contained in $\overline{B(v, 2\sqrt{2}\delta)}$.

Let us show that any point $z \in \Omega$ with  $\dist(z, \partial \Omega) > (5 + \sqrt 2) \delta$ lies in  $\carr \mathcal T$. Evidently,  $z$ lies in a Voronoi cell $F_v$ with $|v - z| \le \sqrt{2}\delta$. By the above discussion, $z$ belong to one of the Delauney triangles $v v_i v_{i+1}$.
\end{proof}

\subsection{Rough quasiconformality}

Let $R$ be a rectangle compactly contained in $\Omega$.
Its {\em exterior discrete approximation} $R^{\disc}_+$  consists of all vertices of $\mathcal T$ that either lie in $R$ or are adjacent to a vertex that lies in $R$. For each corner of $R$, mark the closest point in $\mathcal T \cap R^{\disc}_+$. (In case of a tie, we choose the marked points arbitrarily.) The four marked points turn $R^{\disc}_+$ into a discrete combinatorial rectangle. In practice, $R^{\disc}_+$ is close to $R$: if  $\tilde R$ is a slightly larger rectangle which contains $R$ in its interior, then for $\lambda$ is large, w.h.p.
$ R \, \subset \, R_{\disc}^+ \subset  \tilde R.$

\begin{lemma}
\label{random-delauney-roughly-qc}
For any $\varepsilon > 0$, when $\lambda \ge \lambda_0(\varepsilon, \Omega)$ is sufficiently large, the probability that  $\varphi_{\mathcal P}: \Omega \to \mathbb{D}$ is $(K, \varepsilon)$ roughly quasiconformal on $\Omega$ is  $\ge 1-\varepsilon$.
\end{lemma}

\begin{proof}

Let $R \in \mathcal R_\Omega(\varepsilon)$ be a rectangle compactly contained in $\Omega$ whose sides have length at least $\varepsilon$.
We will estimate the discrete modulus of $\Gamma_{\leftrightarrow}(R^{\disc}_+)$ from above using a discrete metric that is supported on vertices of bounded valence.
By Lemma \ref{moduli-CP}, this would give an upper bound for the continuous modulus of $\Gamma_{\leftrightarrow}(\varphi_{\mathcal P}(R))$, which is what we are after.

Consider the square grid $\delta \mathbb{Z}^2$ with mesh size $\delta = C/\sqrt{\lambda}$.
From the law of large numbers, we expect a cell $\square$ in  $\delta \mathbb{Z}^2$ to contain roughly $C^2$ points from $\mathcal{T}$.
We colour a cell $\square$ in $\delta \mathbb{Z}^2$ {\em blue} if it contains between 1 and $C^3$ points from $\mathcal{T}$ and {\em yellow} otherwise.
We call $\square$ {\em deep blue} if all cells $\hat \square \subset 5 \square$ are blue.
It is easy to see that any vertex of $\mathcal{T}$ in a deep blue cell has valence at most $25 C^3$ since the Delauney edges emanating it from it are contained in $5 \square$.
 By making $C > 0$ large, we can ensure that the probability that a cell is blue is at least $1 - r_0(5)$ where $r_0(5)$ is the constant from modification (iii) of Lemma \ref{mathieu-lemma2}.
Consider the discrete metric $\rho_{\disc} = \chi_{\mathscr B \cap R^{\disc}_+}$
where $\mathscr B$ is the union of the deep blue cells.
By Lemma \ref{points-in-a-square}, if the intensity $\lambda$ is large, then w.h.p.~
$$
A(\rho_{\disc}) \, = \, \sum_v \rho_{\disc}^2(v) \, = \, N_{R^{\disc}_+} \, \le \, 2 \cdot \ell_1(R)\ell_2(R) \lambda \, = \, 2C^2 \cdot \ell_1(R) \ell_2(R)/\delta^2.
$$
To estimate the $\rho_{\disc}$-length of a path $\gamma_{\disc} \in \Gamma_{\leftrightarrow}(R^{\disc}_+)$, we view it as a piecewise linear curve $\gamma$ by connecting the vertices with line segments. According to modification (iii) of Lemma \ref{mathieu-lemma2}, when $\lambda$ is sufficiently large, w.h.p.~every $\gamma$ passes through at least $\lfloor \ell_2(R)/(2\delta) \rfloor$ deep blue cells. By property $(*)$, if $\gamma$ passes through a blue cell, $\gamma_{\disc}$ must contain a vertex in this cell or in one of the eight adjacent cells.
 We see that w.h.p.~the $\rho_{\disc}$-length of every path in $\Gamma_{\leftrightarrow}(R^{\disc}_+)$ is at least $c \cdot \ell_2(R)/\delta$ where $c > 0$ is a definite constant.

The above computations show that when the intensity is large, w.h.p.~the discrete modulus of $\Gamma_{\leftrightarrow}(R^{\disc}_+)$ is bounded above by a definite constant times the continuous modulus of $R$.
Since $\rho_{\disc}$ was supported on vertices of bounded valence, Lemma \ref{moduli-CP} tells us that the continuous modulus of $\Gamma_{\leftrightarrow}(\varphi_{\mathcal P}(R))$ is
bounded by a definite multiple of the continuous modulus of $R$.  This completes the proof.
\end{proof}

\subsection{Interior conformality}

The following lemma says that the radii of interior circles are small when the intensity $\lambda$ is large:

\begin{lemma}
 For any $\varepsilon > 0$ and subdomain $\Omega'$ compactly contained in $\Omega$, if the intensity $\lambda > \lambda_0(\varepsilon, \Omega')$ is sufficiently large, then with probability at least $1-\varepsilon$,
 the radii of all circles $C_v \in \mathcal P$ associated to Delauney points $v \in \Omega'$ are less than $\varepsilon$.
\end{lemma}

\begin{proof}
We can surround a vertex $v \in \Omega'$ by  an annulus $A = A(v, r, r') \subset \Omega'$ of arbitrarily large modulus. When $\lambda$ is large, w.h.p.~all vertices adjacent to $v$ will
lie inside $B(v,r)$.
Rough quasiconformality tells us that that the annulus $\varphi_{\mathcal P}(A)$ will have large modulus. Since the image of $\varphi_{\mathcal P}(A)$ is contained in the unit disk and surrounds $C_v$, the radius of $C_v$ must be small.
\end{proof}

By a deep theorem of He and Schramm \cite[Theorem 1.1]{he-schramm}, we have:

\begin{corollary}
\label{largescale-structure}
Let $S$ be a square compactly contained in $\Omega$ and $\tilde S \subset \Omega$ be a slightly larger square with the same center as $S$. For any $\varepsilon > 0$, when the intensity $\lambda > \lambda_0(\varepsilon, S, \tilde S)$ is sufficiently large, with probability at least $1-\varepsilon$, the
modulus of $\varphi_{\mathcal P}(S)$ is determined by the Delauney triangulation on $\tilde S$ within $\varepsilon$ of its true value.
\end{corollary}

Since the model of random Delauney triangulations does not have a preferred direction, the arguments of Section \ref{sec:homogenization-qc}  show:

\begin{lemma}
Let $\Omega'$ be a subdomain compactly contained in $\Omega$ which contains $z_1, z_2$.
For any $\varepsilon > 0$, when the intensity $\lambda > \lambda_0(\varepsilon, \Omega', \Omega)$ is sufficiently large, with probability at least $1-\varepsilon$, the map $\varphi_{\mathcal P}$ is within $\varepsilon$ of a conformal map defined on $\Omega'$.
\end{lemma}

\subsection{Boundary behaviour}

To complete the proof of Theorem \ref{main-thm4}, we need to show that if $\partial \Omega$ is $C^1$, then the image of the approximating conformal map is the unit disk:

\begin{lemma}
For any $\varepsilon > 0$, there exists an $r > 0$ so that when the intensity $\lambda > \lambda_0(\varepsilon, r)$ is sufficiently large, with probability at least $1-\varepsilon$, the image of $\Omega_r = \{z \in \Omega: \dist(z, \partial \Omega) < r\}$ under $\varphi_{\mathcal P}$ contains $B(0, 1-\varepsilon)$.
\end{lemma}

 In particular, the above lemma implies that all circles in $\mathcal P$
have small radii, not just ones confined to the interior.

\begin{proof}[Sketch of proof]
Since $\partial \Omega$ is $C^1$, there exists a number $\rho_0 > 0$ so that for any boundary point $\zeta \in \partial \Omega$ and $0 < \rho < \rho_0$, the intersection
$\partial B(\zeta, \rho) \cap \Omega$ consists of a single circular arc.

Fix $0 < r' < \rho_0$ so that $\dist(z_1,\partial \Omega) > 2r'$ and
$\dist(z_2,\partial \Omega) > 2r'$.
For a point $\zeta \in \partial \Omega$, consider the conformal rectangle $R = A(\zeta, 2r,r') \cap \Omega$ where the ``flat'' sides have been marked. Since the modulus of $R$ can be made arbitrarily large by making $r$ small, it is reasonable to believe that w.h.p.~its image $\varphi_{\mathcal P}(R)$ also has large modulus. Assuming this temporarily, we see that the diameter of $\varphi_{\mathcal P} \bigl (B( \zeta,2r) \bigr )$ is  small   since  $\varphi_{\mathcal P}(R)$ separates $\varphi_{\mathcal P} \bigl (B(\zeta, 2r) \bigr )$ from $C_{v_1}$ and $C_{v_2}$. The lemma follows since finitely many balls $B(\zeta_i, 2r)$ cover $\Omega \setminus \Omega_r$.

To estimate $\Mod \varphi_{\mathcal P}(R)$, we follow the strategy from the proof of rough quasiconformality (Lemma  \ref{random-delauney-roughly-qc}).
Set $\delta = C/\sqrt{\lambda}$ as in Lemma \ref{random-delauney-roughly-qc}.
Since $\partial \Omega$ is $C^1$, we may partition $\Omega$ into cells of diameter comparable to $\delta$ and area comparable to $\delta^2$ as in Lemma \ref{basic-properties}.
We colour each cell in $\Omega$ either blue and yellow as in Lemma \ref{random-delauney-roughly-qc}, that is, we colour a cell {\em blue} if it contains between 1 and $C^3$ points of $\mathcal T$ and {\em yellow} otherwise.
Consider the metric
$$
\rho_{\disc}(v) = \frac{1}{|v-\zeta|} \cdot \chi_{ \mathscr B \cap R^{\disc}_+},
$$
where $\mathscr B$ is the union of the deep blue cells and $R^{\disc}_+$ is the exterior discrete approximation of $R$.
We claim that w.h.p.~
$$
A(\rho_{\disc}) \lesssim \log \frac{r'}{2r} \cdot (1/\delta)^2,
\qquad
\ell_{\rho_{\disc}}(\gamma) \gtrsim \log \frac{r'}{2r} \cdot (1/\delta), \qquad \gamma \in \Gamma_{\updownarrow}(R^{\disc}_+).
$$
The area estimate follows from the law of large numbers, while the length estimate follows from modification (ii) of Lemma \ref{mathieu-lemma}. These length-area estimates imply that $\Mod \Gamma_{\updownarrow}(R^{\disc}_+) \lesssim
\bigl ( \log \frac{r'}{r} \bigr)^{-1}$. Since $\rho_{\disc}$ is supported on vertices of bounded valence, by Lemma \ref{moduli-CP},
$\Mod \Gamma_{\updownarrow}(\varphi_{\mathcal P}(R)) \lesssim
\bigl ( \log \frac{r'}{r} \bigr)^{-1}$. Hence,
$\Mod\varphi_{\mathcal P}(R) = \Mod \Gamma_{\leftrightarrow}(\varphi_{\mathcal P}(R)) \gtrsim
 \log \frac{r'}{r}$ as desired.
\end{proof}

\appendix

\section[Appendix A. Weak convergence is not enough]{Weak convergence is not enough}

It sounds plausible that if a sequence of Beltrami coefficients $\mu_n$ converges weakly to $\mu$, then the quasiconformal mappings $w^{\mu_n}$ converge pointwise to $w^{\mu}$.
However, this is not true. For a counterexample, partition the plane into vertical strips of width $\delta$  and assign $\mu^\delta =1/3$ on odd-numbered strips and $-1/3$ on even-numbered strips. Clearly, the Beltrami coefficients $\mu_{\delta}$ converge weakly to 0 as $\delta \to 0$, however, the maps $\mu_\delta$ converge to the affine stretch in the horizontal direction
 by the factor $(2+\frac{1}{2})/2 = 5/4$. Indeed, on each even-numbered strip, the $x$-coordinate is stretched by a factor of 2, while in each odd numbered strip, one contracts the $x$-coordinate by a factor of 2.

By the law of large numbers, if one randomly assigns the Beltrami coefficient to be $\pm 1/3$ on vertical strips, then the limit is also an affine stretch by a factor of $5/4$ in the $x$-coordinate.

\end{document}